\documentclass{amsart}
\usepackage{amssymb}
\usepackage{amsmath}
\numberwithin{equation}{section}

\newtheorem{thm}{Theorem}[section]
\newtheorem{lem}[thm]{Lemma}

\newtheorem{conj}[thm]{Conjecture}
\newtheorem{rem}[thm]{Remark}

\theoremstyle{definition}
\newtheorem{defn}[thm]{Definition}

\theoremstyle{remark}

\newcommand{\R}{\mathbb{R}}
\newcommand{\Z}{\mathbb{Z}}

\newcommand{\PP}{\mathbb{P}}

\newcommand{\nm}[1]{{\,\left\|#1\right\|\,}}

\begin{document}

\title{On the Steinhaus tiling problem for  \protect{$\Z^3$}}

\author{Daniel Goldstein}
\address{Center for Communications Research, 4320 Westerra Court, San
  Diego, CA 92121}
\email{danielgolds@gmail.com}
%\thanks{}

\author{R. DANIEL MAULDIN}
%\address{Department of Mathematics\\Box 311430\\University of North Texas\\
%Denton, TX 76203}
\email{mauldin@unt.edu}
%\thanks{Research supported by NSF Grant DMS-0700831}
\thanks{The second-named author was supported by National Science Foundation 
DMS-0700831}
\date{\today}



\begin{abstract} 
Steinhaus asked in the $1950$'s
whether there exists a set in $\R^2$
meeting every isometric copy of $\Z^2$ in precisely one point.
Such a ``Steinhaus set''
was constructed by Jackson and Mauldin. What about $\R^3$?
Is there a subset $S$ of $\mathbb{R}^3$ meeting every 
isometric copy of $\mathbb{Z}^3$ in exactly one point? 
We offer heuristic evidence that the answer is ``no''.

\end{abstract}

%\date{\today}      

\maketitle

\section{Introduction} \label{sec:intro}

Steinhaus asked in the late fifties~\cite[p. 193]{Si}
if there is a set $S \subseteq \R^2$ such that $S$ meets every
isometric copy of $\Z \times \Z$ in exactly one
point. We call such a set $S$ a Steinhaus set.
Equivalently, $S$ is a Steinhaus set if and only if 
$S$ is a simultaneous transversal
for each of the subgroups $R_{\theta}\Z^2$, where
$R_{\theta}=\left(\begin{smallmatrix} \cos\theta  & \sin\theta\\ 
-\sin\theta & \cos\theta\end{smallmatrix}\right)$.
Here a subset $E$ of an abelian group $G$ is a transversal for a
subgroup $H<G$ if $E$ meets every coset of $H$ in $G$ in exactly one point.
This problem was discussed by Croft in~\cite{C} and Croft, Falconer,
and Guy in ~\cite{CFG}, and by Jackson and Mauldin in~\cite{jm3}.

Jackson and Mauldin settled this question in~\cite{jm1} and~\cite{jm2}.

\begin{thm} \label{mainthm}
There is a set $S \subseteq \R^2$ that meets
every isometric copy $\Z \times \Z$ in exactly one point.
\end{thm}

The question of whether the same result holds for $\Z^3$ in $\R^3$
remains unsolved. It is known (see \cite{KW})
that there cannot be a Lebesgue
measurable Steinhaus set in $\R^3$. It is not known whether there can be
a Lebesgue measurable Steinhaus set in $\R^2.$ 
However, it is known that no Steinhaus set in $\R^2$ 
is a Borel set or even has the Baire property~\cite{jm3}.
The analogous questions for $\Z^n$ in $\R^n$ for
$n \geq 4$ have negative answers, as the following observation of
the first-named author shows. If $S \subseteq \R^n$
were a Steinhaus set for $\Z^n$ ($n \geq 4$), let
$x=(a_1,a_1,\dots,a_n) \in S \cap \Z^n$, and let
$y=(b_1+1/2,b_2+1/2,b_3+1/2,b_4+1/2,b_5,\dots,b_n) \in S \cap (\Z^n
+(1/2,1/2,1/2,1/2,0,\dots,0))$. Easily $\nm{x-y}^2 \in \Z$. Since every
non-negative integer is the sum of four squares, this show that the
distance between $x$
and $y$ is the same as a distance between two lattice points in  $\Z^n$.

We investigate here the three-dimensional version of one of the central
results of \cite{jm2}, which involves the idea of extending ``partial
Steinhaus sets'' to larger and larger families of lattices to obtain a
Steinhaus set. We give heuristic evidence that this procedure would
fail, and that partial Steinhaus sets in $\R^3$ don't exist even
for rather small families of lattices. 

Here is an outline of our heuristic. 
If a Steinhaus set exists, then,
a $p$-partial Steinhaus function (see Definition~\ref{defn:mpsf}) exists
for every odd prime $p$.
Let $X_p \subseteq \Z^3$ be the cube
of triples $(a,b,c)$ such that $0\le a,b,c<p$.
By definition, 
a $p$-partial Steinhaus function is a function $L$ from 
$X_p$ to itself satisfying condition $(+)$ of Lemma~\ref{lem:ad}.
Given any function $L$ from $X_p$ to itself we associate $(p+1)p^2$
functions $\pi$ from $GF(p)$ to itself (see Definition~\ref{defn:pi}).
By Theorem~\ref{thm:perms},
$L$ is a $p$-partial Steinhaus function if and only if each of
the $(p+1)p^2$ associated functions is a permutation. 

Here is the numerology.
The cube $X_p$ has $p^3$ points. There are thus $p^{3p^3}$ functions
from $X_p$ to itself. The probability that a random function 
from $\{0,1,\dots, p-1\}$ to itself is a permutation is $p!/p^p$.
If the $p^2(p+1)$ functions associated to $L$ 
were random and independently distributed (which they
are not), the expected number $N_p$ of $p$-partial Steinhaus functions
would be 
$$
N_p = p^{3p^3} \left(\frac{p!}{p^p}\right)^{(p+1)p^2}.
$$

The right-hand side above is less than one for $p=11$ (or for large
$p$ by Stirling's approximation).  However, if the right-hand side
above is less than one, our heuristic argument suggests that no
$p$-partial Steinhaus function exists which, if true, would imply that
no Steinhaus set exists in dimension~$3$.
 
In \cite{jm1} the following
stronger result was shown.

\begin{thm}
There is a set $S \subseteq \R^2$ satisfying:

\begin{enumerate}
\item
$S$ meets every isometric copy of $\Z \times \Z$ in exactly one point.
\item \label{p2s}
If $x,y \in S$ are distinct then $\nm{x-y}^2 \notin \Z$.
\end{enumerate}

\end{thm}

This theorem is stronger since all that was necessary for the 
Steinhaus problem in dimension~2  was that the 
squared distance in (2) not be the square of a lattice distance. Our functions $\pi$
arise from extending the use of similar functions in \cite{jm2} to
prove Theorem 1.2.

The authors thank the anonymous referee for comments that led to
improvements in the manuscript.

\section{Lattice distances} \label{sec:1}
In this paper, a \emph{lattice} in $\R^n$ is an isometric copy of $\Z^n$ in $\R^n$
By definition, an element
of $\{\nm{v}^2 \mid v\in \Z^n\}$ is a \emph{squared lattice distance}.
A squared lattice distance is a sum of $n$ squares of integers.

As already remarked, if $n\ge 4$, then every non-negative integer is a
squared lattice distance. 

The following useful lemma on squared lattice distances applies
in dimension $n=2$ or $3$.
Weil \cite[p. 292]{weil} attributes it to L.~Aubry~\cite{A}.
See math overflow~\cite{mo} for a nice discussion.

\begin{lem} \label{lem:3squares}
Let $n=2$ or $3$. If an integer $N$ is a sum of $n$ rational squares
then $N$ is a sum of $n$ integer squares (i.e.\ is a squared lattice distance).
\end{lem}

\begin{proof} We treat the more difficult case $n=3$. The case $n=2$
is treated similarly (and is not used in this paper).

Find a common denominator $m$ for the three  rational numbers so that
\begin{equation} \label{eq:lat}
N = \left(\frac{a}{m}\right)^2 + \left(\frac{b}{m}\right)^2 + \left(\frac{c}{m}\right)^2
\end{equation}
where $a,b,c$ and $m$ are integers.
We argue by contradiction. Assume that $N$ is \emph{not} the sum of three
squares of integers. Choose a solution to \eqref{eq:lat} with $m>0$ minimal.
By assumption $m>1$.

For the rest of the proof we work in three-dimensional Euclidean space.
The point $P = (\frac{a}{m},\frac{b}{m},\frac{c}{m})$ 
lies on the sphere $S$ centered at the origin $O$ 
with squared radius $N$ by \eqref{eq:lat}. 
Let $Z\in \Z^3$ be the nearest lattice point to $P$ 
(choose any if there is more than one). 
The observation
$$
0 < \nm{P-Z}^2 \le 3/4 <1.
$$
has two curious consequences. (1) the line
$\overline{PZ}$ 
is not tangent to $S$ (else $OPZ$ would be a right triangle in which
case $\nm{O-Z}^2 = \nm{O-P}^2 + \nm{P-Z}^2$. This is contradiction since 
the left hand side is an integer and the right hand side is not.)

It follows from (1) that the line $\overline{PZ}$ meets $S$ in two
points, $P$ and say $P'$. The second consequence is that $P'$ 
has rational coordinates
and yields a solution to \eqref{eq:lat} with denominator $m'$ strictly smaller than $m$. 
Indeed, we can write   
$$
P' = tP + (1-t)Z 
$$
for some real number $t$.
Then $\nm{P'}^2 = \nm{Z + t(P-Z)}^2 = \nm{Z}^2 + tZ\cdot(P-Z) + t^2\nm{P-Z}^2 .$ 
Since $P'$ and $Z$ both lie on $S$ their squared lengths are equal, so that 
$0 = tZ\cdot(P-Z) + t^2\nm{P-Z}^2,$ whence 
$$
t = -  Z\cdot(P-Z) / \nm{P-Z}^2.
$$
From this we see that the point $m\nm{Z-P}^2P'$ has integer coordinates.
But $m\nm{Z-P}^2$ is a positive integer strictly less than $m$. This 
contradiction proves the lemma.
\end{proof}

\section{Partial Steinhaus sets} \label{sec:2}
We begin this section with the definition of $m$-partial Steinhaus sets
and some elementary remarks.
For $m >1$ an integer, set 
$$
X_m= \{ ( a,b,c )  \mid a,b,c \in \Z\quad\mathrm{and}\quad 0 \leq a,b,c <m \}.
$$

\begin{defn} \label{def:partial}
A subset $S\subseteq \R^3$ is an 
\emph{$m$-partial Steinhaus set} 
if (i) \mbox{$|S \cap (\frac1m{x} + \Z^3)|=1$} for all $x\in X_m$, and (ii) 
$\nm{x-z}^2$ is not a squared lattice distance for all distinct
$x,z\in S$.
\end{defn}

\begin{rem} \label{rem:1}

\begin{enumerate}
\item 
Note that by (i), $S$ is of the form $\{\frac1mx+L(x) \mid x \in
X_m\}$ for some function $L:X_m \to \Z^3.$
\item By a straightforward calculation, if $L(x)$ furnishes an $m$-partial Steinhaus set
then so also does $L'(x) = L(x) -  m\gamma(x)$ for any function 
$\gamma:X_m\to \Z^3$.
\item  If an $m$-partial Steinhaus set exists and 
$m'\mid m$, 
then an $m'$-partial Steinhaus set exists.

\end{enumerate}

\end{rem}

We use  repeatedly the fact that any vector $v\in\Z^3$ is 
uniquely written:
\begin{equation} \tag{*} \label{eq:vecreduce}
v = y(v) + m\epsilon(v)
\end{equation}
with $y(v)\in X_m$ and $\epsilon(v)\in \Z^3$.
For each coordinate $1\le i\le 3$, 
$\epsilon(v)_i$ is the quotient and $y(v)_i$ 
the remainder when $v_i$ is divided by $m.$

From this fact and Remark~\ref{rem:1}(2) we are free to replace 
$L(x)$ by $y(L(x))$. Thus if an $m$-Steinhaus set exists, we may (and do)
assume $L:X_m\to X_m$.

Since $\nm{L(x)-L(z)}^2$ is a squared lattice distance if it is an integer by 
Lemma~\ref{lem:3squares}, we have proved:

\begin{lem} \label{lem:ad} We have
\begin{enumerate}
\item
For $m >1$ be an integer, let $L:X_m\to\Z^3$ be any function.
The set  $\{\frac1mx + L(x) \mid x \in X_m\}$ is an $m$-partial
Steinhaus set if and only if 

\begin{equation} \tag{+}
\nm{\left(\frac1mz + L(z)\right) - \left(\frac1mx + L(x)\right)}^2 
\quad \text{ \ is not an integer} 
\end{equation}
for any two distinct elements $x$ and $z$ of $X_m$.
\item If so, then $\{\frac1mx + L(x) + mg(x) \mid x \in X_m \}$ 
is an $m$-partial Steinhaus set for any $g:X_m\to\Z^3$. In particular,
taking $g(x) =-\epsilon(L(x))$,  the set $\{\frac1mx + y(L(x))
\mid x \in X_m \}$ 
is an $m$-partial Steinhaus set.
\end{enumerate}
\end{lem}

\begin{defn} \label{defn:mpsf}
We call a function $L:X_m\to X_m$ satisfying (+) an
$m$-partial Steinhaus function.
\end{defn}

We analyze whether there is such a function $L$ when $m$ is a prime
greater than~$2$. 

\begin{conj} \label{conj:p}
For $p$ a sufficiently large prime, there does not exist
a $p$-partial Steinhaus set in $\R^3$.
\end{conj}

The truth of Conjecture~\ref{conj:p} would imply that the Steinhaus
problem in $\R^3$ has a negative solution. 
We give heuristic evidence for Conjecture~\ref{conj:p} in
section~\ref{sec:3}.

We observe that, whatever $L$ is, 
if $\nm{(\frac1pz+ L(z))-(\frac1px + L(x))}^2$ is an integer, then 
we must have that $\nm{z-x}^2$ is divisible by $p$.

We make the following conventions.
Let $\Lambda\subseteq X_p$ consist of triples $(x,y,z) \ne (0,0,0)$ 
whose squared norm is
divisible by $p$ (and we note that Lemma~\ref{lem:count} counts the
number of such triples). For $\lambda \in \Lambda$, since 
$\nm{\lambda}^2$ is divisible by $p$, there is a unique 
$d=d(\lambda)$ in $\{0,...,p-1\}$ such that 
$\nm{\lambda}^2 \equiv dp \pmod{p^2}.$ 

\begin{defn} \label{defn:pi}
Let  $L:X_p\to X_p$ be any function. 
Set $F(x) = \frac1px + L(x).$
For each
$\lambda \in \Lambda$ and $x \in X_p$, we define a mapping 
$\pi^\lambda_x$ from
$\{0,\dots, p-1 \}$ to itself:
\begin{equation} \label {pxl}
\pi^\lambda_x(t) = \frac{td(\lambda)}{2} + \lambda\cdot
[L(y(x+\lambda t))-\epsilon(x+\lambda t)]\pmod p.
\end{equation}
\end{defn}

For convenience, if $\lambda$ and $x$ are fixed, we shorten the
notation 
to $y_t=y(x+\lambda t)$ and $\epsilon_t= \epsilon(x+\lambda t).$

Let $GF(p)$ denote the field of integers modulo $p$.
Throughout we identify the elements of $GF(p)$ with the set
$\{0,1,\cdots, p-1\}$.

\begin{thm} \label{thm:equiv}
For $p$ an odd prime and $L:X_p\to X_p$ any function,
the following three statements are equivalent:

\begin{enumerate}
\item For all $x,z \in X_p$ with $x\neq z$,
$$
\nm{F(z) -F(x)}^2 \notin \Z.
$$

\item For all  $\lambda \in \Lambda$, all $x\in X_p$,
and all distinct $s,t\in \{0,\cdots p-1\}$, 
$$
\nm{F(y_t) -F(y_s)}^2 \notin \Z.
$$

\item For all $\lambda \in \Lambda$ and for all $x\in X_p$,
the function  $\pi^\lambda_x$
is a permutation of $GF(p)$.
\end{enumerate}
\end{thm}

Note that the relationship between partial Steinhaus sets and
permutations was already exploited in~\cite{jm2}.

\begin{proof} 
For the equivalence of (1) and (2) it suffices to note that if
$x,z\in X_p$ are distinct and $\nm{F(z) -F(x)}^2 \in \Z $ then 
$z-x=  \lambda\pmod p$ for some $\lambda\in \Lambda,$ thus $z=y_1$ and $x = y_0$ in the notation
defined above.

Fix $\lambda \in \Lambda$ and $x\in X_p$. For $0 \leq s,t
  <p$ we have

\begin{align*}
&\nm{F(y_t)  -F(y_s)}^2   \notin \Z\\
\iff &\nm{\left(\frac{y_t}{p}+L(y_t)\right)-\left(\frac{y_s}{p} - L(y_s)\right)}^2   \notin \Z\\
\iff &\nm{\left(y_t - y_s\right) + p(L(y_t)-L(y_s)}^2 \notin p^2\Z \\ 
\iff& \nm{(t-s)\lambda +p[(L(y_t)-\epsilon_t)
   -(L(y_s)-\epsilon_s)]}^2\notin p^2\Z\\
\iff &(t-s)^2dp +2(t-s)p\lambda\cdot[(L(y_t)-\epsilon_t)-(L(y_s)-\epsilon_s)]
 \not\in p^2\Z\\
\iff &(t-s)^2d +2(t-s)\lambda\cdot[(L(y_t)-\epsilon_t)-(L(y_s)-\epsilon_s)]
\not\in p\Z\\
\iff &2(t-s) (\pi_x^\lambda(t) - \pi_x^\lambda(s)) \not\in p\Z.\\
\end{align*}

Since $p$ is odd, if $s\ne t$ then $2(t-s)$ is coprime to $p$, whence 

$$\nm{F(y_t)  -F(y_s)}^2 \notin \Z \iff 
\pi_x^\lambda(t)  \not\equiv \pi_x^\lambda(s) \pmod p.
$$

This proves the equivalence of (2) and (3).
\end{proof}

If there is a $p$-partial Steinhaus set in dimension~$3$, then
Theorem~\ref{thm:equiv} yields a family
of permutations of $GF(p)$ denoted 
$\pi_x^\lambda$, indexed by $x\in X_p$ and $\lambda\in \Lambda$. 
There are further conditions that this family must satisfy.

\begin{lem} \label{lem:conditions}
The permutations $\pi_x^\lambda$ satisfy:
\begin{enumerate}
\item 
For $a \in \{0,1,\cdots, p-1\}$, 
$$
\pi_x^\lambda(t+a)= \frac{ad}{2} + \pi_{x+a\lambda}^\lambda(t) \pmod p.
$$

\item 
For $\alpha \in GF(p)$, 
$$
\pi_x^\lambda(\alpha t)=\pi_{x}^{\alpha \lambda}(t) \pmod p
.$$
\end{enumerate}
\end{lem}

\begin{proof}
Part (1) follows from plugging
the identity $x+\lambda(t+a)= (x+a\lambda) + \lambda t$ into
\eqref{eq:vecreduce} and using the definition of $\pi$.
Part (2) follows similarly from the identity  $x+\lambda(\alpha t)= x
+(\alpha\lambda) t,$ 
noting that $\alpha\lambda\in \Lambda$ since $\lambda\in \Lambda$. 
\end{proof}

The questions arise as to how many permutations $\pi^\lambda_x$ are
needed and what conditions they must satisfy in order that a $p$-partial
Steinhaus function $L$ exists. In Theorem~3.10, we show that a
particular set of $(p+1)p^2$ permutations suffice. Towards showing this, we first note
the following. 

Let $\PP^2(GF(p))$ be the projective plane over $GF(p)$.

\begin{lem} \label{lem:count}
Let $p$ be a prime, $p>2$. Then 
\begin{enumerate}
\item There are exactly $p+1$ triples $(\alpha,
\beta,\gamma) \in \PP^2(GF(p))$ such that 
$\alpha^2 + \beta^2 +\gamma^2 \equiv 0.$
\item 
The set $\Lambda = \{(a,b,c) \mid 0 \leq a,b,c < p,
a^2+b^2+c^2 \equiv  0  \pmod p, (a,b,c) \ne(0,0,0)$ has cardinality
$|\Lambda| = p^2-1$.
\end{enumerate}
\end{lem}

\begin{proof}
Let $P=(\alpha,\beta,\gamma) \in \PP^2(GF(p))$ be a triple
satisfying $x^2 + y^2 +z^2 \equiv 0 \pmod p$.
Such a triple exists since the sets  $\{1+\beta^2 \mid \beta
\in GF(p)\}$ and $\{-\gamma^2: \gamma \in GF(p)\}$, each of cardinality 
$\frac{p+1}{2}$ have nonempty intersection by the pigeon-hole principle.

Note that the curve $C:x^2 + y^2 +z^2 \equiv 0 \pmod p$ contains $P$
and has degree $2$.
There is one line $\ell_t$ through $P$ for each slope $t\in\PP^1(GF(p))$.
Since $C$ has degree $2$, each line $\ell_t$
meets $C$ in two points: $P$ and say $Q_t$
(where $Q_t=P$ iff $\ell_t$ is tangent to $C$). 
Conversely, for each point $Q$ of $C$ 
there is a unique line meeting $C$ in $P$ and $Q$.
Thus, the set of points of $C$ is in bijection with 
the set of lines through $P$. There are $p+1$ possible slopes $t$
whence $p+1$ lines whence $p+1$ points on $C$.
The line through $P$ with slope $t, 0\le t<p$  meets $C$ in the point
$$
(\alpha t^2-2\beta t -\alpha, -\beta t^2 -2\alpha t + \beta,
\gamma(1+t^2)).
$$
The vertical line through $P$ meets $C$ in the point
$(\alpha,-\beta,\gamma).$
(Note also that $t=-\alpha/\beta$ is the slope of the tangent line and
we get back $P$ again.)

This is precisely the set of $p+1$ triples in $\PP^2(GF(p))$ 
that satisfy $x^2+y^2+z^2 = 0 \in GF(p)$.
This proves (i). 

For (ii), each projective triple
gives rise to
$p-1$ triples in $\Lambda$ for a total of $(p-1)(p+1)= p^2-1$ by part (i). 
\end{proof}

To state the next lemma, it is convenient
to choose a subset $W\subset \Lambda$ of cardinality $p+1$
representing each projective solution exactly once.
By the proof of Lemma~\ref{lem:count}, we have seen that we may take 
$$
W = \{(\alpha,-\beta,\gamma)\} \cup 
\{(\alpha t^2-2\beta t -\alpha, -\beta t^2 -2\alpha t + \beta,
\gamma(1+t^2))  \mid t\in GF(p)\}.
$$

Each $\lambda\in\Lambda$ determines by definition a one-dimensional
subspace $\ell_{\lambda}$ of
$GF(p)^3$. Let $C_{\lambda}$ be a complementary subspace to
$\ell_{\lambda}$.
Identifying $GF(p)$ with $\{0,1,\cdots, p-1\}$, we view $C_{\lambda}$ as 
a subset of $X_p$. Its cardinality is $p^2$.

\begin{thm} \label{thm:perms}
Let $p>2$ be a prime. Let $L:X_p\to X_p$ be any function.
Then the following are equivalent:

\begin{enumerate}

\item The set $\{\frac1px + L(x) \mid x \in X_p\}$ is a 
$p$-partial Steinhaus set.

\item 
For all $\lambda \in W$ and $x\in X_{\lambda}$ the map $
\pi^{\lambda}_x$ is a permutation.
\end{enumerate}
\end{thm}

Note that (2) furnishes a collection of 
$(p+1)p^2$ maps from $GF(p)$ to $GF(p)$ that are
required to be permutations.

\begin{proof}
This follows from Lemma~\ref{lem:ad}, Theorem~\ref{thm:equiv}, 
Lemma~\ref{lem:conditions}. 
\end{proof} 
\section{Special values of~$p$} \label{sec:3}

In this section we study the number $N_p$ of $p$-partial Steinhaus
sets, where $p$ is a prime. We show that $N_3$ is positive,
and give a heuristic argument that $N_p$ is $0$ for $p\geq 11$. 
This heuristic argument, if it were in fact valid, would imply a
negative
solution to the Steinhaus problem in $\R^3$.

Recall the definition of a $p$-partial Steinhaus set.
Let $L$ be an arbitrary function from a $X_p$ to itself. 
From $L$ is constructed $(p+1)p^2$  subsidiary functions
$\pi^{\lambda}_x$ from $GF(p)$ to $GF(p)$, and $L$ determines a
$p$-partial Steinhaus set if and only if all of the 
subsidiary functions are permutations.

We show that there is a $3$-partial Steinhaus set (function) in dimension~$3$.
The function $L$ is determined by its $81=3\cdot 3^3$ coordinate
functions in $GF(3)$.
In fact, using Theorem~3.10 we find a $3$-partial Steinhaus set where each of the 
$36 = (3+1)3^2$ associated functions from $GF(3)$ to itself 
is an \emph{even} permutation (not just a permutation).
Note that a sufficient condition for a function $\pi:GF(3)\to GF(3)$ to be
an even permutation is 
$$\pi(1)=\pi(0)+1 \quad \rm{and}  \quad \pi(2)=\pi(1)+1.
$$
The condition that a given associated function is even is equivalent
to two linear conditions. 
It is easy (by computer!) to solve this.  The $36$ associated functions 
give rise to~$72$ equations in~$81$ variables. 
A computer calculation easily finds the following $3$-partial
Steinhaus set.
$$
\begin{array}{lll}
\left( 1, 2, 2 \right)&\left( 2, 2, 1/3 \right)&\left( 1, 2, 5/3 \right)\\
\left( 2, 1/3, 2 \right)&\left( 0, 1/3, 1/3 \right)&\left( 2, 1/3, 5/3 \right)\\
\left( 1, 5/3, 2 \right)&\left( 2, 5/3, 1/3 \right)&\left( 1, 5/3, 5/3 \right)\\
\left( 1/3, 2, 2 \right)&\left( 7/3, 2, 1/3 \right)&\left( 1/3, 2, 5/3 \right)\\
\left( 7/3, 1/3, 2 \right)&\left( 4/3, 1/3, 1/3 \right)&\left( 7/3, 1/3, 5/3 \right)\\
\left( 1/3, 5/3, 2 \right)&\left( 7/3, 5/3, 1/3 \right)&\left( 1/3, 5/3, 5/3 \right)\\
\left( 2/3, 0, 0 \right)&\left( 2/3, 0, 1/3 \right)&\left( 2/3, 0, 2/3 \right)\\
\left( 2/3, 1/3, 0 \right)&\left( 2/3, 1/3, 1/3 \right)&\left( 2/3, 1/3, 2/3 \right)\\
\left( 2/3, 2/3, 0 \right)&\left( 2/3, 2/3, 1/3 \right)&\left( 2/3, 2/3, 2/3 \right)\\
\end{array}
$$
What can one expect for other values of $p$? At the time of writing
we don't know whether $N_5>0$.

We give a heuristic argument. 
Let $L$ be a function from $X_p$ to itself.
Associated to any such $L$ are some functions $\pi^\lambda_x$ 
from $GF(p)$ to itself.
Then $L$ is a $p$-partial Steinhaus function if and only if 
each of the associated functions from  $GF(p)$ to itself is a 
permutation.

Since $X_p$ has cardinality $|X_p| = p^3$, there are 
$(p^3)^{p^3}$ functions $L$ from $X_p$ to itself.
The probability that a random function from $GF(p)$ to itself is a
permutation is  $p!/p^p$.

If the set of $(p+1)p^2$ functions $\pi^\lambda_x$  was chosen
uniformly from the collection of all  $(p+1)p^2$-tuples of functions
from $GF(p)$ to itself,
the expected value of
$N_p$ would be 
$$ 
M_p := p^{3p^3} \left(\frac{p!}{p^p}\right)^{(p+1)p^2}
$$

Here is a table  of $M_p$ for some small values of $p$.

$$
\begin{array}{r|l}
p & M_p\\
\hline
3 &1.4 \ E15\\
5 &5.8 \ E49\\
7 & 100 \\
11 & 1.1 \ E-1438\\
13 & 4.0 \ E-3748\\
\end{array}
$$

By Stirling's approximation, $\log(M_p) = -p^4 + 3.5p^3\log(p)\ +$
lower order terms. 
This heuristic reasoning offers some evidence that $N_p = 0$ for large $p$
(even $p\geq 11$). Thus, there would be no $p$-partial Steinhaus
function, and the Steinhaus problem in three dimensions would have a
negative solution.
See~\cite{hj} for some further work on the Steinhaus problem in two dimensions
by the NSF-RTG group at the University of North Texas.

\end{document}